\documentclass{article}

\usepackage[latin1]{inputenc}
\usepackage{mathdots,graphicx,amsfonts,amssymb,amsmath,amsthm,mathrsfs,epstopdf,enumitem,mathtools,color,caption,booktabs}
\usepackage[top=2cm,bottom=3cm,left=2cm,right=2cm]{geometry}

\newtheorem{theorem}{Theorem}[section]

\newtheorem{lemma}{Lemma}[section]
\newtheorem{corollary}{Corollary}[section]
\theoremstyle{definition}
\newtheorem{definition}{Definition}[section]
\newtheorem{remark}{Remark}[section]
\newtheorem{example}{Example}[section]

\numberwithin{equation}{section}
\numberwithin{figure}{section}
\numberwithin{table}{section}

\renewcommand{\epsilon}{\varepsilon}
\newcommand{\xx}{{\boldsymbol x}}
\newcommand{\uu}{{\boldsymbol u}}
\newcommand{\vv}{{\boldsymbol v}}

\begin{document}

\title{Spectral properties of flipped Toeplitz matrices}

\author{Giovanni Barbarino\\
\footnotesize Mathematics and Operational Research Unit, University of Mons, Belgium (giovanni.barbarino@umons.ac.be)\\[10pt]
Sven-Erik Ekstr\"om\\
\footnotesize Division of Scientific Computing, Department of Information Technology, Uppsala University, Sweden (sven-erik.ekstrom@uu.se)\\[10pt]
Carlo Garoni\\
\footnotesize Department of Mathematics, University of Rome Tor Vergata, Italy (garoni@mat.uniroma2.it)\\[10pt]
David Meadon\\
\footnotesize Division of Scientific Computing, Department of Information Technology, Uppsala University, Sweden (david.meadon@uu.se)\\[10pt]
Stefano Serra-Capizzano\\
\footnotesize Department of Science and High Technology, University of Insubria, Italy (s.serracapizzano@uninsubria.it)\\
\footnotesize Division of Scientific Computing, Department of Information Technology, Uppsala University, Sweden (stefano.serra@it.uu.se)\\[10pt]
Paris Vassalos\\
\footnotesize Department of Informatics, Athens University of Economics and Business, Greece (pvassal@aueb.gr)}
\date{}

\maketitle

\begin{abstract}
We study the spectral properties of flipped Toeplitz matrices of the form $H_n(f)=Y_nT_n(f)$, where $T_n(f)$ is the $n\times n$ Toeplitz matrix generated by the function $f$ and $Y_n$ is the $n\times n$ exchange (or flip) matrix having $1$ on the main anti-diagonal and $0$ elsewhere. In particular, under suitable assumptions on $f$, we establish an alternating sign relationship between the eigenvalues of $H_n(f)$, the eigenvalues of $T_n(f)$, and the quasi-uniform samples of $f$.
Moreover, after fine-tuning a few known theorems on Toeplitz matrices, we use them to provide localization results for the eigenvalues of $H_n(f)$. Our study is motivated by the convergence analysis of the minimal residual (MINRES) method for the solution of real non-symmetric Toeplitz linear systems of the form $T_n(f)\mathbf x=\mathbf b$ after pre-multiplication of both sides by $Y_n$, as suggested by Pestana and Wathen.

\smallskip


\smallskip

\end{abstract}

\section{Introduction}
A matrix of the form
\begin{equation}\label{mbtm_expr}
\left[f_{i-j}\right]_{i,j=1}^{n}=\begin{bmatrix}
f_0 & f_{-1} & \ \cdots & \ \ \cdots & f_{-(n-1)} \\
f_1 & \ddots & \ \ddots & \ \ & \vdots\\
\vdots & \ddots & \ \ddots & \ \ddots & \vdots\\
\vdots & & \ \ddots & \ \ddots & f_{-1}\\[3pt]
f_{n-1} & \cdots & \ \cdots & \ \ f_1 & f_0
\end{bmatrix},
\end{equation}
whose entries are constant along each diagonal, is called a Toeplitz matrix. 
In the case where the entries $f_k$ are the Fourier coefficients of a function $f:[-\pi,\pi]\to\mathbb C$ in $L^1([-\pi,\pi])$, i.e.,
\[ f_k=\frac1{2\pi}\int_{-\pi}^\pi f(x){\rm e}^{-{\rm i}kx}{\rm d}x,\qquad k\in\mathbb Z, \]
the matrix \eqref{mbtm_expr} is denoted by $T_n(f)$ and is referred to as the $n$th Toeplitz matrix generated by $f$.

The efficient solution of a linear system with a coefficient matrix of the form $T_n(f)$ by means of Krylov subspace methods is a research topic that involved several researchers over time. The main efforts focused on the case where $T_n(f)$ is real symmetric positive definite, so that the conjugate gradient (CG) method can be applied as well as its preconditioned version. Whenever $T_n(f)$ is real symmetric but indefinite, an alternative to (preconditioned) CG is the (preconditioned) minimal residual (MINRES) method. A common feature of CG and MINRES is that their convergence bounds rely only on the eigenvalues of the system matrix, or on the eigenvalues of the preconditioned system matrix if preconditioning 
is applied; see \cite[Section~2.2]{Fabio-book} and keep in mind that MINRES is (mathematically) equivalent to the generalized minimal residual (GMRES) method for real symmetric matrices.
In the case where the Fourier coefficients of $f$ are real but $T_n(f)$ is not symmetric, Pestana and Wathen \cite{PW} suggested pre-multiplying $T_n(f)$ by the $n\times n$ exchange (or flip) matrix
\begin{equation}\label{flipM}
Y_n=\begin{bmatrix*}
&&&&1\\
&&&1\\
&&\iddots\\
&1\\
1
\end{bmatrix*}.
\end{equation}
In this way, the resulting flipped matrix $H_n(f)=Y_nT_n(f)$ is real symmetric and (preconditioned) MINRES can be applied. 

A main reason behind the interest in the spectral properties of flipped Toeplitz matrices such as $H_n(f)$ is precisely the convergence analysis of MINRES. 
In this regard, a precise asymptotic spectral distribution theorem for the sequence of flipped Toeplitz matrices $\{H_n(f)\}_n$ was established independently by Ferrari et al.~\cite{SIMAX-hankel} through techniques based on the notion of approximating classes of sequences \cite[Chapter~5]{GLTbookI}, and by Mazza and Pestana~\cite{BIT-hankel} through the theory of block generalized locally Toeplitz sequences \cite{GLTbookIII}. The same kind of study was later extended to flipped multilevel Toeplitz matrices in \cite{ELA-multihankel,SIMAX-multihankel}. However, no 
localization result for the eigenvalues of $H_n(f)$ was provided so far in the literature, despite the importance of spectral localization in the convergence analysis of MINRES. 

In this paper, based on classical results for Toeplitz matrices \cite{BoSi,SC-matrix,GLTbookI} and on recent results on the asymptotic spectral distribution of arbitrary sequences of matrices \cite{sdau}, we delve deeper into the spectral properties of $H_n(f)$ under suitable assumptions on the function $f$. In particular:
\begin{itemize}[nolistsep,leftmargin=*]
	\item We show that $\lceil n/2\rceil$ eigenvalues of $H_n(f)$ coincide with $\lceil n/2\rceil$ eigenvalues of $T_n(f)$ and have an asymptotic distribution described by $f$, while the other $\lfloor n/2\rfloor$ eigenvalues of $H_n(f)$ coincide with $\lfloor n/2\rfloor$ eigenvalues of $-T_n(f)$ and have an asymptotic distribution described by $-f$.
	\item As an extension of the previous result, we show that, for every $n$, the eigenvalues of $H_n(f)$ are given by the following alternating sign relationship:
	\[ \lambda_i(H_n(f))=(-1)^{i+1}\lambda_i(T_n(f))=(-1)^{i+1}f(x_{i,n})+\epsilon_{i,n},\qquad i=1,\ldots,n, \]
	where $\lambda_i(T_n(f))$, $i=1,\ldots,n$, are the eigenvalues of $T_n(f)$, $\{x_{i,n}\}_{i=1,\ldots,n}$ is an asymptotically uniform grid (see Section~\ref{a.u.}), and $\max_{i=1,\ldots,n}|\epsilon_{i,n}|\to0$ as $n\to\infty$; moreover, $\epsilon_{i,n}=0$ for all $i=1,\ldots,n$ if $f$ has a finite number of local maximum/minimum points and discontinuities.
	\item After fine-tuning a few known theorems on Toeplitz matrices, we use them to provide localization results for the eigenvalues of $H_n(f)$. 
\end{itemize}
The paper is organized as follows. Section~\ref{prel} contains some preliminaries. Section~\ref{main} contains statements and proofs of our main spectral results for flipped Toeplitz matrices of the form $H_n(f)$. Section~\ref{num_exps} contains numerical experiments that illustrate some of the main results. Section~\ref{conc} contains final remarks.

\section{Preliminaries}\label{prel}

\subsection{Notation and Terminology}\label{a.u.}
We denote by $\mu_d$ the Lebesgue measure in $\mathbb R^d$. Throughout this paper, all the terminology from measure theory (such as ``measurable'', ``a.e.'', etc.)\ always refers to the Lebesgue measure. The closure of a set $E$ is denoted by $\overline E$.
We use a notation borrowed from probability theory to indicate sets. For example, if $f,g:\Omega\subseteq\mathbb R^d\to\mathbb R$, then $\{f\le1\}=\{\xx\in\Omega:f(\xx)\le1\}$, $\mu_d\{f>0,\:g<0\}$ is the measure of the set $\{\xx\in\Omega:f(\xx)>0,\:g(\xx)<0\}$, etc.

Given a measurable function $f:\Omega\subseteq\mathbb R^d\to\mathbb C$, the essential range of $f$ is denoted by $\mathcal{ER}(f)$. We recall that $\mathcal{ER}(f)$ is defined as
\[ \mathcal{ER}(f)=\{z\in\mathbb C:\hspace{0.5pt}\mu_d\{|f-z|<\epsilon\}>0\hspace{0.5pt}\mbox{ for all }\hspace{0.5pt}\epsilon>0\}. \]
It is clear that $\mathcal{ER}(f)\subseteq\overline{f(\Omega)}$.
Moreover, $\mathcal{ER}(f)$ is closed and $f\in\mathcal{ER}(f)$ a.e.; see, e.g., \cite[Lemma~2.1]{GLTbookI}.
If $f$ is real a.e.\ then $\mathcal{ER}(f)$ is a subset of $\mathbb R$. In this case, we define the essential infimum (resp., supremum) of $f$ on $\Omega$ as the infimum (resp., supremum) of $\mathcal{ER}(f)$:
\[ \mathop{\rm ess\,inf}_\Omega f=\inf\mathcal{ER}(f),\qquad\mathop{\rm ess\,sup}_\Omega f=\sup\mathcal{ER}(f). \]

Throughout this paper, any finite sequence of points in $\mathbb R$ is referred to as a grid.
Consider an interval $[a,b]$ and, for every $n$, let $\mathcal G_n=\{x_{i,n}\}_{i=1,\ldots,d_n}$ be a grid of $d_n$ points in $[a,b]$ 
with $d_n\to\infty$ as $n\to\infty$. The number
\begin{equation*}
m(\mathcal G_n)=\max_{i=1,\ldots,d_n}\left|x_{i,n}-\Bigl(a+i\,\frac{b-a}{d_n+1}\Bigr)\right|
\end{equation*}
measures the distance of $\mathcal G_n$ from the uniform grid $\{a+i(b-a)/(d_n+1)\}_{i=1,\ldots,d_n}$; we refer to it as the uniformity measure of the grid $\mathcal G_n$.
We say that $\mathcal G_n$ is asymptotically uniform (a.u.)\ in $[a,b]$ if
\[ \lim_{n\to\infty}m(\mathcal G_n)=0. \]

\subsection{Asymptotic singular value and eigenvalue distributions of a matrix-sequence}
Throughout this paper, a matrix-sequence is a sequence of the form $\{A_n\}_n$, where $A_n$ is a square matrix and ${\rm size}(A_n)=d_n\to\infty$ as $n\to\infty$.
We denote by $C_c(\mathbb R)$ (resp., $C_c(\mathbb C)$) the space of continuous complex-valued functions with bounded support defined on $\mathbb R$ (resp., $\mathbb C$).
If $A\in\mathbb C^{m\times m}$, the singular values and eigenvalues of $A$ are denoted by $\sigma_1(A),\ldots,\sigma_m(A)$ and $\lambda_1(A),\ldots,\lambda_m(A)$, respectively. The minimum and maximum singular values of $A$ are also denoted by $\sigma_{\min}(A)$ and $\sigma_{\max}(A)$.
A matrix-valued function $f:\Omega\subseteq\mathbb R^d\to\mathbb C^{r\times r}$ is said to be measurable (resp., bounded, continuous, continuous a.e., in $L^p(\Omega)$, etc.)\ if its components $f_{ij}:\Omega\to\mathbb C$, $i,j=1,\ldots,r$, are measurable (resp., bounded, continuous, continuous a.e., in $L^p(\Omega)$, etc.).

\begin{definition}[\textbf{asymptotic singular value and eigenvalue distributions of a matrix-sequence}]\label{dd}
Let $\{A_n\}_n$ be a matrix-sequence with $A_n$ of size $d_n$, and let $f:\Omega\subset\mathbb R^d\to\mathbb C^{k\times k}$ be measurable with $0<\mu_d(\Omega)<\infty$.
\begin{itemize}[nolistsep,leftmargin=*]
	\item We say that $\{A_n\}_n$ has an asymptotic eigenvalue (or spectral) distribution described by $f$ if
	\begin{equation}\label{sd}
	\lim_{n\to\infty}\frac1{d_n}\sum_{i=1}^{d_n}F(\lambda_i(A_n))=\frac1{\mu_d(\Omega)}\int_\Omega\frac{\sum_{i=1}^kF(\lambda_i(f(\xx)))}{k}\mathrm{d}\xx,\qquad\forall\,F\in C_c(\mathbb C).
	\end{equation}
	In this case, $f$ is called the eigenvalue (or spectral) symbol of $\{A_n\}_n$ and we write $\{A_n\}_n\sim_\lambda f$.
	\item We say that $\{A_n\}_n$ has an asymptotic singular value distribution described by $f$ if
	\begin{equation}\label{svd}
	\lim_{n\to\infty}\frac1{d_n}\sum_{i=1}^{d_n}F(\sigma_i(A_n))=\frac1{\mu_d(\Omega)}\int_\Omega\frac{\sum_{i=1}^kF(\sigma_i(f(\xx)))}{k}\mathrm{d}\xx,\qquad\forall\,F\in C_c(\mathbb R).
	\end{equation}
	In this case, $f$ is called the singular value symbol of $\{A_n\}_n$ and we write $\{A_n\}_n\sim_\sigma f$.
\end{itemize}
\end{definition}

We remark that Definition~\ref{dd} is well-posed as the functions $\xx\mapsto\sum_{i=1}^kF(\lambda_i(f(\xx)))$ and $\xx\mapsto\sum_{i=1}^kF(\sigma_i(f(\xx)))$ appearing in \eqref{sd}--\eqref{svd} are measurable \cite[Lemma~2.1]{GLTbookIII}. Throughout this paper, whenever we write a relation such as $\{A_n\}_n\sim_\lambda f$ or $\{A_n\}_n\sim_\sigma f$, it is understood that $\{A_n\}_n$ and $f$ are as in Definition~\ref{dd}, i.e., $\{A_n\}_n$ is a matrix-sequence and $f$ is a measurable function taking values in $\mathbb C^{k\times k}$ for some $k$ and defined on a subset $\Omega$ of some $\mathbb R^d$ with $0<\mu_d(\Omega)<\infty$. 
Since any finite multiset of numbers can always be interpreted as the spectrum of a matrix, a byproduct of Definition~\ref{dd} is the following definition.

\begin{definition}[\textbf{asymptotic distribution of a sequence of finite multisets of numbers}]\label{adn}
Let $\{\Lambda_n=\{\lambda_{1,n},\ldots,\lambda_{d_n,n}\}\}_n$ be a sequence of finite multisets of numbers such that $d_n\to\infty$ as $n\to\infty$, and let $f$ be as in Definition~\ref{dd}. 
We say that $\{\Lambda_n\}_n$ has an asymptotic distribution described by $f$, and we write $\{\Lambda_n\}_n\sim f$, if $\{A_n\}_n\sim_\lambda f$, where $A_n$ is any matrix whose spectrum equals $\Lambda_n$ (e.g., $A_n={\rm diag}(\lambda_{1,n},\ldots,\lambda_{d_n,n})$).
\end{definition}

The next lemma is a slight generalization of \cite[Lemma~3.12]{sdau} and it can be proved in the same way.

\begin{lemma}\label{lem:concat}
Let $\{A_n\}_n$ be a matrix-sequence, let $f:[a,b]\to\mathbb C^{r\times r}$ be measurable, and suppose that $\{A_n\}_n\sim_\lambda f$. 
Let $\lambda_1(f),\ldots,\lambda_r(f):[a,b]\to\mathbb C$ be $r$ measurable functions such that $\lambda_1(f(x)),\ldots,\lambda_r(f(x))$ are the eigenvalues of $f(x)$ for every $x\in[a,b]$. Then, for every $\alpha,\beta\in\mathbb R$ with $\alpha<\beta$, we have $\{A_n\}_n\sim_\lambda\tilde f$, where $\tilde f$ is the following concatenation on the interval $[\alpha,\beta]$ of resized versions of $\lambda_1(f),\ldots,\lambda_r(f):$
\begin{equation*}
\tilde f:[\alpha,\beta]\to\mathbb C,\qquad \tilde f(y)=\left\{\begin{aligned}
&\textstyle{\lambda_1(f(a+\frac{(b-a)r}{\beta-\alpha}(y-\alpha))),} &&\textstyle{\alpha\le y<\alpha+\frac{\beta-\alpha}r,}\\[5pt]
&\textstyle{\lambda_2(f(a+\frac{(b-a)r}{\beta-\alpha}(y-\alpha-\frac{\beta-\alpha}{r}))),} &&\textstyle{\alpha+\frac{\beta-\alpha}r\le y<\alpha+2\frac{\beta-\alpha}r,}\\[5pt]
&\textstyle{\lambda_3(f(a+\frac{(b-a)r}{\beta-\alpha}(y-\alpha-2\frac{\beta-\alpha}{r}))),} &&\textstyle{\alpha+2\frac{\beta-\alpha}r\le y<\alpha+3\frac{\beta-\alpha}r,}\\[5pt]
&\qquad\vdots && \qquad\vdots\\[5pt]
&\textstyle{\lambda_r(f(a+\frac{(b-a)r}{\beta-\alpha}(y-\alpha-(r-1)\frac{\beta-\alpha}{r}))),} &&\textstyle{\alpha+(r-1)\frac{\beta-\alpha}r\le y\le\beta.}\\
\end{aligned}\right.
\end{equation*}
\end{lemma}

Theorems~\ref{eccoGi-n}--\ref{thm:mmd} are fundamental asymptotic distribution results obtained in \cite{sdau}. They play a central role hereinafter. 
Throughout this paper, we use ``increasing'' as a synonym of ``non-decreasing''.

\begin{theorem}\label{eccoGi-n}
Let $f:[a,b]\to\mathbb R$ be bounded and continuous a.e.\ with $\mathcal{ER}(f)=[\inf_{[a,b]}f,\sup_{[a,b]}f]$. 
Let $\{\Lambda_n=\{\lambda_{1,n},\ldots,\lambda_{d_n,n}\}\}_n$ be a sequence of finite multisets of 
real numbers such that $d_n\to\infty$ as $n\to\infty$. Assume the following. 
\begin{itemize}[nolistsep,leftmargin=*]
	\item $\{\Lambda_n\}_n\sim f$. 
	\item $\Lambda_n\subseteq[\inf_{[a,b]}f-\epsilon_n,\sup_{[a,b]}f+\epsilon_n]$ for every $n$ and for some $\epsilon_n\to0$ as $n\to\infty$.
\end{itemize}
Then, for every a.u.\ grid $\{x_{i,n}\}_{i=1,\ldots,d_n}$ in $[a,b]$, if $\sigma_n$ and $\tau_n$ are two permutations of $\{1,\ldots,d_n\}$ such that the vectors $[f(x_{\sigma_n(1),n}),\ldots,f(x_{\sigma_n(d_n),n})]$ and $[\lambda_{\tau_n(1),n},\ldots,\lambda_{\tau_n(d_n),n}]$ are sorted in increasing order, we have
\[ \max_{i=1,\ldots,d_n}|f(x_{\sigma_n(i),n})-\lambda_{\tau_n(i),n}|\to0\ \,\mbox{as}\ \,n\to\infty. \]
In particular,
\[ \min_\tau\max_{i=1,\ldots,d_n}|f(x_{i,n})-\lambda_{\tau(i),n}|\to0\ \,\mbox{as}\ \,n\to\infty, \]
where the minimum is taken over all permutations $\tau$ of $\{1,\ldots,d_n\}$.
\end{theorem}

To properly state Theorem~\ref{thm:mmd}, we need the following definition.

\begin{definition}[\textbf{local extremum points}]\label{wlep}
Given a function $f:[a,b]\to\mathbb R$ and a point $x_0\in[a,b]$, we say that $x_0$ is a local maximum point (resp., local minimum point) for $f$ if $f(x_0)\ge f(x)$ (resp., $f(x_0)\le f(x)$) for all $x$ belonging to a neighborhood of $x_0$ in $[a,b]$.
\end{definition}

We point out that, according to Definition~\ref{wlep}, by ``local maximum/minimum point'' we mean ``weak local maximum/minimum point''.
For example, if $f$ is constant on $[a,b]$, then all points of $[a,b]$ are both local maximum and local minimum points for $f$.

\begin{theorem}\label{thm:mmd}
Let $f:[a,b]\to\mathbb R$ be bounded with a finite number of local maximum points, local minimum points, and discontinuity points, and with $\mathcal{ER}(f)=[\inf_{[a,b]}f,\sup_{[a,b]}f]$. Let $\{\Lambda_n=\{\lambda_{1,n},\ldots,\lambda_{d_n,n}\}\}_n$ be a sequence of finite multisets of real numbers such that $d_n\to\infty$ as $n\to\infty$. Assume the following. 
\begin{itemize}[nolistsep,leftmargin=*]
	\item $\{\Lambda_n\}_n\sim f$. 
	\item $\Lambda_n\subseteq f([a,b])$ for every $n$.
\end{itemize}
Then, for every $n$, there exist an a.u.\ grid $\{x_{i,n}\}_{i=1,\ldots,d_n}$ in $[a,b]$ and a permutation $\tau_n$ of $\{1,\ldots,d_n\}$ such that
\[ \lambda_{\tau_n(i),n}=f(x_{i,n}),\qquad i=1,\ldots,d_n. \]
\end{theorem}

\subsection{Toeplitz matrices}

It is not difficult to see that the operator $T_n(\cdot):L^1([-\pi,\pi])\to\mathbb C^{n\times n}$, which associates with each $f\in L^1([-\pi,\pi])$ the corresponding $n\times n$ Toeplitz matrix $T_n(f)$, is linear and satisfies $T_n(1)=I_n$, where $I_n$ is the $n\times n$ identity matrix. Moreover, the conjugate transpose of $T_n(f)$ is given by
\begin{equation*}
T_n(f)^*=T_n(\overline f)
\end{equation*}
for every $f\in L^1([-\pi,\pi])$ and every $n$; see, e.g., \cite[Section~6.2]{GLTbookI}.
In particular, if $f$ is real a.e., then $\overline f=f$ a.e.\ and the matrices $T_n(f)$ are Hermitian. Moreover, if $f$ is real a.e.\ and even, then its Fourier coefficients are real and even (see Lemma~\ref{feven} below), and therefore the matrices $T_n(f)$ are real and symmetric.
The next theorem collects some properties of Toeplitz matrices generated by a real function.

\begin{theorem}\label{Toep-th}
Let $f\in L^1([-\pi,\pi])$ be real and let
\[ m_f=\mathop{\rm ess\,inf}_{[-\pi,\pi]}f,\qquad M_f=\mathop{\rm ess\,sup}_{[-\pi,\pi]}f. \]
Then, the following properties hold.
\begin{enumerate}[nolistsep,leftmargin=*]
	\item $T_n(f)$ is Hermitian and the eigenvalues of $T_n(f)$ lie in the interval $[m_f,M_f]$ for all $n$.
	\item If $f$ is not a.e.\ constant, then $m_f<M_f$ and the eigenvalues of $T_n(f)$ lie in $(m_f,M_f)$ for all $n$.
	\item $\{T_n(f)\}_n\sim_\lambda f$.
\end{enumerate}
\end{theorem}
\begin{proof}
See \cite[Theorems~6.1 and~6.5]{GLTbookI}.
\end{proof}

The next localization result for the singular values of Toeplitz matrices is due to Widom \cite[Lemma~I.2]{Widom}.

\begin{lemma}\label{Widom's}
Let $f\in L^1([-\pi,\pi])$ and let $d$ be the distance of the complex zero from the convex hull of the essential range $\mathcal{ER}(f)$. Then, the singular values of $T_n(f)$ lie in $[d,M_{|f|}]$ for all $n$, where $M_{|f|}=\mathop{\rm ess\,sup}_{[-\pi,\pi]}|f|$.
\end{lemma}

\subsection{Flipped Toeplitz matrices}
If $f:[-\pi,\pi]\to\mathbb C$ is a function in $L^1([-\pi,\pi])$ and $n\in\mathbb N$, we define the Hankel matrix
\[ H_n(f)=Y_nT_n(f), \]
where $Y_n=H_n(1)$ is the $n\times n$ exchange matrix in \eqref{flipM}. We remark that our definition of $H_n(f)$ is different from the standard definition of Hankel matrices generated by a function $f$ \cite[Section~11.4]{BoSi}. We therefore refer to $H_n(f)$ as the flipped Toeplitz matrix generated by $f$ rather than the Hankel matrix generated by $f$.
A vector $\vv\in\mathbb R^n$ is called symmetric if $Y_n\vv=\vv$ and skew-symmetric if $Y_n\vv=-\vv$.

\begin{remark}[\textbf{eigendecomposition of the exchange matrix}]\label{antidiagonal_eig}
Let $\mathbb V_n^+$ and $\mathbb V_n^-$ be the subspaces of $\mathbb R^n$ consisting of symmetric and skew-symmetric vectors, respectively:
\[ \mathbb V_n^+=\{\vv\in\mathbb R^n:Y_n\vv=\vv\},\qquad\mathbb V_n^-=\{\vv\in\mathbb R^n:Y_n\vv=-\vv\}. \]
It is easy to see that $\dim\mathbb V_n^+=\lceil n/2\rceil$ and $\dim\mathbb V_n^-=\lfloor n/2\rfloor$. Indeed, a basis for $\mathbb V_n^+$ is
\[ \boldsymbol e_i+\boldsymbol e_{n-i+1},\qquad i=1,\ldots,\lceil n/2\rceil, \]
and a basis for $\mathbb V_n^-$ is
\[ \boldsymbol e_i-\boldsymbol e_{n-i+1},\qquad i=1,\ldots,\lfloor n/2\rfloor, \]
where $\boldsymbol e_1,\ldots,\boldsymbol e_n$ are the vectors of the canonical basis of $\mathbb R^n$.
Note that $\mathbb V_n^+$ and $\mathbb V_n^-$ are, by definition, eigenspaces of $Y_n$ associated with the eigenvalues $1$ and $-1$, respectively, and we have $\dim\mathbb V_n^++\dim\mathbb V_n^-=n$. This yields the eigendecomposition of the exchange matrix $Y_n$, which has only two distinct eigenvalues $1$ and $-1$ with corresponding eigenspaces $\mathbb V_n^+$ and $\mathbb V_n^-$. We can thus write the eigendecomposition of $Y_n$ as follows:
\begin{equation}\label{exchange_eig}
Y_n=V_n\Delta_nV_n^{-1},\qquad\Delta_n=\mathop{\rm diag}_{i=1,\ldots,n}(-1)^{i+1},\qquad V_n=\left[\vv_1\,|\,\vv_2\,|\,\cdots\,|\,\vv_n\right],
\end{equation}
where $\{\vv_i:i\mbox{ is odd}\}$ is any basis of $\mathbb V_n^+$ and $\{\vv_i:i\mbox{ is even}\}$ is any basis of $\mathbb V_n^-$.
\end{remark}

If $f:[-\pi,\pi]\to\mathbb C$ is a function in $L^1([-\pi,\pi])$ with real Fourier coefficients, then $T_n(f)$ is real for all $n$.
In this case, $H_n(f)$ is real and symmetric for every $n$, and the next theorem appeared in \cite{SIMAX-hankel,BIT-hankel} gives the asymptotic spectral distribution of the matrix-sequence $\{H_n(f)\}_n$.

\begin{theorem}\label{r-hankel}
If $f$ is a function in $L^1([-\pi,\pi])$ with real Fourier coefficients then $\{H_n(f)\}_n\sim_\lambda H$, where
\[ H:[-\pi,\pi]\to\mathbb C^{2\times2},\qquad H(x)={\rm diag}(f(x),-f(x)). \]
\end{theorem}

The next lemma is a corollary of the Cantoni--Butler theorem \cite[Theorem~2]{SC-matrix}.
An $n\times n$ matrix $A$ is called centrosymmetric if it is symmetric with respect to its center, i.e., $A_{ij}=A_{n-i+1,n-j+1}$ for all $i,j=1,\ldots,n$.
Equivalently, $A$ is centrosymmetric if $Y_nAY_n = A$. Note that any symmetric Toeplitz matrix is centrosymmetric.

\begin{lemma}\label{cbt}
Let $T_n$ be a real symmetric Toeplitz matrix of size $n$ and let $H_n=Y_nT_n$. Then, the following properties hold.
\begin{enumerate}[leftmargin=*,nolistsep]
	\item There exists an orthonormal basis of $\mathbb R^n$ consisting of eigenvectors of $T_n$ such that $\lceil n/2\rceil$ vectors of this basis are symmetric and the other $\lfloor n/2\rfloor$ vectors are skew-symmetric.
	\item Let $\{\vv_1,\ldots,\vv_n\}$ be a basis of $\mathbb R^n$ such that, setting $V_n=\left[\vv_1\,|\,\vv_2\,|\,\cdots\,|\,\vv_n\right]$, we have:
	\begin{itemize}[leftmargin=*,nolistsep]
		\item $\vv_i$ is alternatively symmetric or skew-symmetric (starting with symmetric), i.e., in view of Remark~{\rm\ref{antidiagonal_eig}},
		\[ Y_n=V_n\Delta_nV_n^*,\qquad\Delta_n=\mathop{\rm diag}_{i=1,\ldots,n}(-1)^{i+1}; \]
		\item $T_n\vv_i=\lambda_i(T_n)\vv_i$ for $i=1,\ldots,n$, i.e.,
		\[ T_n=V_nD_nV_n^*,\qquad D_n=\mathop{\rm diag}_{i=1,\ldots,n}\lambda_i(T_n). \]
	\end{itemize}
	Then, 
	\begin{itemize}[nolistsep,leftmargin=*]
	\item $H_n\vv_i=\lambda_i(H_n)\vv_i$ with $\lambda_i(H_n)=(-1)^{i+1}\lambda_i(T_n)$ for $i=1,\ldots,n$, i.e.,
	\[ H_n=V_nE_nV_n^*,\qquad E_n=\Delta_nD_n. \]
	\end{itemize}
\end{enumerate}
\end{lemma}
\begin{proof}
Since $T_n$ is symmetric centrosymmetric, the first property follows from \cite[Theorem~2]{SC-matrix}. The second property is a consequence of the first property and the definition $H_n=Y_nT_n$.
\end{proof}

Taking into account that the matrices $T_n(f)$ are real and symmetric whenever $f$ is real and even, the following result is a corollary of Lemma~\ref{cbt} and Theorem~\ref{Toep-th}.

\begin{corollary}\label{cb}
Let $f\in L^1([-\pi,\pi])$ be real and even. Then, for every $n$, there exist a real unitary matrix $V_n$ and an ordering of the eigenvalues of $T_n(f)$ and $H_n(f)$ such that
\begin{alignat*}{3}
Y_n&=V_n\Delta_nV_n^*, &\qquad\Delta_n&=\mathop{\rm diag}_{i=1,\ldots,n}(-1)^{i+1},\\
T_n(f)&=V_nD_nV_n^*, &\qquad D_n&=\mathop{\rm diag}_{i=1,\ldots,n}\lambda_i(T_n(f)),\\
H_n(f)&=V_nE_nV_n^*, &\qquad E_n&=\Delta_nD_n=\mathop{\rm diag}_{i=1,\ldots,n}\lambda_i(H_n(f)),\\
\lambda_i(H_n(f))&=(-1)^{i+1}\lambda_i(T_n(f)), &\qquad i&=1,\ldots,n.
\end{alignat*}
In particular, if $f$ is not a.e.\ constant, the eigenvalues of $H_n(f)$ lie in $(m_f,M_f)\cup(-M_f,-m_f)$, where
\[ m_f=\mathop{\rm ess\,inf}_{[-\pi,\pi]}f,\qquad M_f=\mathop{\rm ess\,sup}_{[-\pi,\pi]}f. \]
\end{corollary}

\section{Main results: spectral properties of flipped Toeplitz matrices}\label{main}

In this section, we state and prove the main results of this paper. Theorem~\ref{cb+-} is our first main result.
Throughout this paper, if $\Lambda$ is a multiset of numbers and $\alpha\in\mathbb C$, we denote by $\Lambda+\alpha$ and $\Lambda-\alpha$ the multisets $\{\lambda+\alpha:\lambda\in\Lambda\}$ and $\{\lambda-\alpha:\lambda\in\Lambda\}$, respectively.

\begin{theorem}\label{cb+-}
Let $f\in L^1([-\pi,\pi])$ be real and even, and suppose that $m_f=\mathop{\rm ess\,inf}_{[0,\pi]}f>-\infty$ or $M_f=\mathop{\rm ess\,sup}_{[0,\pi]}f<\infty$. Then, for every $n$, there exist a real unitary matrix $V_n$ and an ordering of the eigenvalues of $T_n(f)$ and $H_n(f)$ such that
\begin{alignat}{3}
Y_n&=V_n\Delta_nV_n^*, &\qquad\Delta_n&=\mathop{\rm diag}_{i=1,\ldots,n}(-1)^{i+1},\label{(i)}\\
T_n(f)&=V_nD_nV_n^*, &\qquad D_n&=\mathop{\rm diag}_{i=1,\ldots,n}\lambda_i(T_n(f)),\label{(ii)}\\
H_n(f)&=V_nE_nV_n^*, &\qquad E_n&=\Delta_nD_n=\mathop{\rm diag}_{i=1,\ldots,n}\lambda_i(H_n(f)),\label{(iii)}\\
\lambda_i(H_n(f))&=(-1)^{i+1}\lambda_i(T_n(f)), &\qquad i&=1,\ldots,n,\label{(iv)}\\
\{\Lambda_n^+\}_n&\sim f,&\qquad\Lambda_n^+&=\{\lambda_{2i-1}(H_n(f))\}_{i=1,\ldots,\lceil n/2\rceil}=\{\lambda_{2i-1}(T_n(f))\}_{i=1,\ldots,\lceil n/2\rceil},\label{l+}\\
\{\Lambda_n^-\}_n&\sim-f,&\qquad\Lambda_n^-&=\{\lambda_{2i}(H_n(f))\}_{i=1,\ldots,\lfloor n/2\rfloor}=\{-\lambda_{2i}(T_n(f))\}_{i=1,\ldots,\lfloor n/2\rfloor}.\label{l-}
\end{alignat}
\end{theorem}
\begin{proof}
We prove the theorem in the case where $m_f>-\infty$ (the proof in the case where $M_f<\infty$ is similar).
Let $g=f+\alpha$, where $\alpha\in\mathbb R$ is a constant such that $m_f+\alpha>0$. Note that $g$ is a real even function in $L^1([-\pi,\pi])$ like~$f$, and moreover $g\ge m_f+\alpha>0$ a.e.\ in $[0,\pi]$, so $\mathcal{ER}(g)$ is a closed subset of $(0,\infty)$.
By Corollary~\ref{cb}, for every~$n$, there exist a real unitary matrix $V_n$ and an ordering of the eigenvalues of $T_n(g)$ and $H_n(g)$ such that 
\begin{alignat}{3}
Y_n&=V_n\Delta_nV_n^*, &\qquad\Delta_n&=\mathop{\rm diag}_{i=1,\ldots,n}(-1)^{i+1},\label{(i)g}\\
T_n(g)&=V_nD_n'V_n^*, &\qquad D_n'&=\mathop{\rm diag}_{i=1,\ldots,n}\lambda_i(T_n(g)),\label{(ii)g}\\
H_n(g)&=V_nE_n'V_n^*, &\qquad E_n'&=\Delta_nD_n'=\mathop{\rm diag}_{i=1,\ldots,n}\lambda_i(H_n(g)),\label{(iii)g}\\
\lambda_i(H_n(g))&=(-1)^{i+1}\lambda_i(T_n(g)), &\qquad i&=1,\ldots,n.\label{(iv)g}
\end{alignat}
Let
\begin{align*}
\tilde\Lambda_n^+&=\{\lambda_{2i-1}(H_n(g))\}_{i=1,\ldots,\lceil n/2\rceil}=\{\lambda_{2i-1}(T_n(g))\}_{i=1,\ldots,\lceil n/2\rceil},\\
\tilde\Lambda_n^-&=\{\lambda_{2i}(H_n(g))\}_{i=1,\ldots,\lfloor n/2\rfloor}=\{-\lambda_{2i}(T_n(g))\}_{i=1,\ldots,\lfloor n/2\rfloor}.
\end{align*}
We prove that $\{\tilde\Lambda_n^+\}_n\sim g$. For every $F\in C_c(\mathbb C)$, let $\tilde F\in C_c(\mathbb C)$ be a function such that $\tilde F=F$ on $[m_f+\alpha,\infty)$ and $\tilde F=0$ on $(-\infty,-m_f-\alpha]$. Note that $\tilde\Lambda_n^+\subseteq[m_f+\alpha,\infty)$ and $\tilde\Lambda_n^-\subseteq(-\infty,-m_f-\alpha]$ by Theorem~\ref{Toep-th}.
By Theorem~\ref{r-hankel} and the evenness of $g$,
\begin{align*}
\lim_{n\to\infty}\frac1{\lceil n/2\rceil}\sum_{i=1}^{\lceil n/2\rceil}F(\lambda_{2i-1}(H_n(g)))&=\lim_{n\to\infty}\frac1{\lceil n/2\rceil}\sum_{i=1}^{\lceil n/2\rceil}\tilde F(\lambda_{2i-1}(H_n(g)))=\lim_{n\to\infty}\frac n{\lceil n/2\rceil}\cdot\frac1n\sum_{i=1}^{n}\tilde F(\lambda_i(H_n(g)))\\
&=2\cdot\frac1{\pi}\int_0^\pi\frac{\tilde F(g(x))+\tilde F(-g(x))}2{\rm d}x=\frac1{\pi}\int_0^\pi\tilde F(g(x)){\rm d}x=\frac1{\pi}\int_0^\pi F(g(x)){\rm d}x.
\end{align*}
Hence, $\{\tilde\Lambda_n^+\}_n\sim g$. Similarly, one can show that $\{\tilde\Lambda_n^-\}_n\sim-g$.

To conclude the proof, we note that, by \eqref{(i)g}--\eqref{(ii)g} and the linearity of the operator $T_n(\cdot)$,
\begin{align}
T_n(f)&=T_n(g-\alpha)=T_n(g)-\alpha I_n=V_n(D_n'-\alpha I_n)V_n^*,\label{iuppif0}\\
H_n(f)&=Y_nT_n(f)=V_n\Delta_n(D_n'-\alpha I_n)V_n^*.\label{iuppi'f0} 
\end{align}
Define the following ordering for the eigenvalues of $T_n(f)$ and $H_n(f)$:
\begin{align}
\lambda_i(T_n(f))&=\lambda_i(T_n(g))-\alpha,\qquad i=1,\ldots,n,\label{iuperf0}\\
\lambda_i(H_n(f))&=(-1)^{i+1}(\lambda_i(T_n(g))-\alpha)=(-1)^{i+1}\lambda_i(T_n(f)),\label{iuper'f0} 
\qquad i=1,\ldots,n.
\end{align}
In view of \eqref{iuppif0}--\eqref{iuppi'f0}, it is now easy to check that \eqref{(i)}--\eqref{(iv)} are satisfied with
\begin{align*}
D_n&=D_n'-\alpha I_n,\\
E_n&=\Delta_n(D_n'-\alpha I_n)=\Delta_nD_n.
\end{align*}
Moreover, also \eqref{l+}--\eqref{l-} are satisfied, because $\{\tilde\Lambda_n^+\}_n\sim g$ is equivalent to $\{\Lambda_n^+\}_n\sim f$ and $\{\tilde\Lambda_n^-\}_n\sim-g$ is equivalent to $\{\Lambda_n^-\}_n\sim-f$. These equivalences follow from Definition~\ref{adn}, the equation $g=f+\alpha$, and the observation that
\begin{align*}
\tilde\Lambda_n^+&=\{\lambda_{2i-1}(T_n(g))\}_{i=1,\ldots,\lceil n/2\rceil}=\{\lambda_{2i-1}(T_n(f))+\alpha\}_{i=1,\ldots,\lceil n/2\rceil}=\Lambda_n^++\alpha,\\
\tilde\Lambda_n^-&=\{-\lambda_{2i}(T_n(g))\}_{i=1,\ldots,\lceil n/2\rceil}=\{-\lambda_{2i}(T_n(f))-\alpha\}_{i=1,\ldots,\lceil n/2\rceil}=\Lambda_n^--\alpha. \tag*{\qedhere}
\end{align*}
\end{proof}

Theorem~\ref{thm3-Giov-vts} is our second main result.

\begin{theorem}\label{thm3-Giov-vts}
Let $f:[-\pi,\pi]\to\mathbb R$ be even, bounded and continuous a.e.\ with $\mathcal{ER}(f)=[\inf_{[0,\pi]}f,\sup_{[0,\pi]}f]$. Then, for every $n$ and every a.u.\ grid $\{x_{i,n}\}_{i=1,\ldots,n}$ in $[0,\pi]$, there exist a real unitary matrix $V_n$ and an ordering of the eigenvalues of $T_n(f)$ and $H_n(f)$ such that
\begin{alignat}{3}
Y_n&=V_n\Delta_nV_n^*,&\qquad\Delta_n&=\mathop{\rm diag}_{i=1,\ldots,n}(-1)^{i+1},\label{a-f}\\
T_n(f)&=V_nD_nV_n^*,&\qquad D_n&=\mathop{\rm diag}_{i=1,\ldots,n}\lambda_i(T_n(f)),\label{a--f}\\
H_n(f)&=V_nE_nV_n^*,&\qquad E_n&=\Delta_nD_n=\mathop{\rm diag}_{i=1,\ldots,n}\lambda_i(H_n(f)),\label{a---f}\\
\lambda_i(H_n(f))&=(-1)^{i+1}\lambda_i(T_n(f)), &\qquad i&=1,\ldots,n,\label{ef}\\
\max_{i=1,\ldots,n}&|f(x_{i,n})-\lambda_i(T_n(f))|\to0\ \,\mbox{as}\,\ n\to\infty.\label{af}
\end{alignat}
\end{theorem}
\begin{proof}
By Theorem~\ref{cb+-}, for every $n$ there exist a real unitary matrix $V_n$ and an ordering of the eigenvalues of $T_n(f)$ and $H_n(f)$ such that \eqref{a-f}--\eqref{ef} are satisfied and
\begin{alignat}{3}
\{\Lambda_n^+\}_n&\sim f,&\qquad\Lambda_n^+&=\{\lambda_{2i-1}(H_n(f))\}_{i=1,\ldots,\lceil n/2\rceil}=\{\lambda_{2i-1}(T_n(f))\}_{i=1,\ldots,\lceil n/2\rceil},\label{l+f}\\
\{\Lambda_n^-\}_n&\sim-f,&\qquad\Lambda_n^-&=\{\lambda_{2i}(H_n(f))\}_{i=1,\ldots,\lfloor n/2\rfloor}=\{-\lambda_{2i}(T_n(f))\}_{i=1,\ldots,\lfloor n/2\rfloor}.\label{l-f}
\end{alignat}
Note that if we permute the columns of $V_n$ and the eigenvalues of $T_n(f)$ and $H_n(f)$ through a same permutation $\tau$ of $\{1,\ldots,n\}$ such that $\tau$ maps odd indices to odd indices and even indices to even indices, then \eqref{a-f}--\eqref{ef} 
continue to hold.

By \eqref{l+f}--\eqref{l-f}, Theorem~\ref{Toep-th}, and the assumptions on $f$, the hypotheses of Theorem~\ref{eccoGi-n} are satisfied for $f$ and $\Lambda_n^+$ as well as for $-f$ and $\Lambda_n^-$.
Hence, by Theorem~\ref{eccoGi-n} applied first with $f$ and $\Lambda_n^+$ and then with $-f$ and $\Lambda_n^-$, we infer that, for every pair of a.u.\ grids $\{x^+_{i,n}\}_{i=1,\ldots,\lceil n/2\rceil}$, $\{x^-_{i,n}\}_{i=1,\ldots,\lfloor n/2\rfloor}$ in $[0,\pi]$, we have
\begin{alignat}{3}
\min_\tau\max_{i=1,\ldots,\lceil n/2\rceil}|f(x_{i,n}^+)-\lambda_{2\tau(i)-1}(T_n(f))|\to0\ \,\mbox{as}\, \ n\to\infty,\label{sub1}\\
\min_\tau\max_{i=1,\ldots,\lfloor n/2\rfloor}|-f(x_{i,n}^-)+\lambda_{2\tau(i)}(T_n(f))|\to0\ \,\mbox{as}\, \ n\to\infty,\label{sub2}
\end{alignat}
where the minima are taken over all permutations $\tau$ of $\{1,\ldots,\lceil n/2\rceil\}$ and $\{1,\ldots,\lfloor n/2\rfloor\}$, respectively.

Now let $\{x_{i,n}\}_{i=1,\ldots,n}$ be an a.u.\ grid in $[0,\pi]$. The two subgrids $\{x_{2i-1,n}\}_{i=1,\ldots,\lceil n/2\rceil}$, $\{x_{2i,n}\}_{i=1,\ldots,\lfloor n/2\rfloor}$ are a.u.\ in $[0,\pi]$. We can therefore use these subgrids in \eqref{sub1}--\eqref{sub2} and we obtain
\begin{alignat}{3}
\min_\tau\max_{i=1,\ldots,\lceil n/2\rceil}|f(x_{2i-1,n})-\lambda_{2\tau(i)-1}(T_n(f))|\to0\ \,\mbox{as}\, \ n\to\infty,\label{sub1'}\\
\min_\tau\max_{i=1,\ldots,\lfloor n/2\rfloor}|-f(x_{2i,n})+\lambda_{2\tau(i)}(T_n(f))|\to0\ \,\mbox{as}\, \ n\to\infty.\label{sub2'}
\end{alignat}
For every $n$, we rearrange the eigenvalues of $H_n(f)$ and $T_n(f)$ as follows:
\begin{align*}
\{\lambda_1(H_n(f)),\ldots,\lambda_n(H_n(f))\}&=\{\lambda_{2\tau_n^+(1)-1}(H_n(f)),\lambda_{2\tau_n^-(1)}(H_n(f)),\lambda_{2\tau_n^+(2)-1}(H_n(f)),\lambda_{2\tau_n^-(2)}(H_n(f)),\ldots\},\\
\{\lambda_1(T_n(f)),\ldots,\lambda_n(T_n(f))\}&=\{\lambda_{2\tau_n^+(1)-1}(T_n(f)),\lambda_{2\tau_n^-(1)}(T_n(f)),\lambda_{2\tau_n^+(2)-1}(T_n(f)),\lambda_{2\tau_n^-(2)}(T_n(f)),\ldots\},
\end{align*}
where $\tau_n^+$ and $\tau_n^-$ are two permutations for which the minima in \eqref{sub1'}--\eqref{sub2'} are attained. Then, \eqref{sub1'}--\eqref{sub2'} imply
\begin{equation*}
\max_{i=1,\ldots,n}|f(x_{i,n})-\lambda_i(T_n(f))|\to0\ \,\mbox{as}\,\ n\to\infty,
\end{equation*}
which yields \eqref{af}.
Moreover, after the above rearrangement of the eigenvalues of $T_n(f)$ and $H_n(f)$,
\eqref{a-f}--\eqref{ef} continue to hold, because, by construction, the considered rearrangement is associated with a permutation $\tau_n$ of the columns of $V_n$ and the eigenvalues of $T_n(f)$ and $H_n(f)$ such that $\tau_n$ maps odd indices to odd indices and even indices to even indices. Of course, 
\eqref{a-f}--\eqref{ef} continue to hold with a new matrix $V_n$ obtained by permuting the columns of the old $V_n$ through the permutation $\tau_n$. With abuse of notation, we denote again by $V_n$ the new matrix~$V_n$, so that 
\eqref{a-f}--\eqref{ef} hold unchanged. The thesis is proved.
\end{proof}

Theorem~\ref{0notin} is our third main result.

\begin{theorem}\label{0notin}
Let $f:[-\pi,\pi]\to\mathbb R$ be even and bounded with a finite number of local maximum points, local minimum points, and discontinuity points, and with $\mathcal{ER}(f)=[\inf_{[0,\pi]}f,\sup_{[0,\pi]}f]$ and $(\inf_{[0,\pi]}f,\sup_{[0,\pi]}f)\subseteq f([0,\pi])$.
Then, for every $n$, there exist an a.u.\ grid $\{x_{i,n}\}_{i=1,\ldots,n}$ in $[0,\pi]$, a real unitary matrix $V_n$, and an ordering of the eigenvalues of $T_n(f)$ and $H_n(f)$ such that
\begin{alignat}{3}
Y_n&=V_n\Delta_nV_n^*,&\qquad\Delta_n&=\mathop{\rm diag}_{i=1,\ldots,n}(-1)^{i+1},\label{a-}\\
T_n(f)&=V_nD_nV_n^*,&\qquad D_n&=\mathop{\rm diag}_{i=1,\ldots,n}\lambda_i(T_n(f)),\label{a--}\\
H_n(f)&=V_nE_nV_n^*,&\qquad E_n&=\Delta_nD_n=\mathop{\rm diag}_{i=1,\ldots,n}\lambda_i(H_n(f)),\label{a---}\\
\lambda_i(H_n(f))&=(-1)^{i+1}\lambda_i(T_n(f)), &\qquad i&=1,\ldots,n,\label{e}\\
\lambda_i(T_n(f))&=f(x_{i,n}), &\qquad i&=1,\ldots,n.\label{a}
\end{alignat}
\end{theorem}
\begin{proof}
By Theorem~\ref{cb+-}, for every $n$ there exist a real unitary matrix $V_n$ and an ordering of the eigenvalues of $T_n(f)$ and $H_n(f)$ such that \eqref{a-}--\eqref{e} are satisfied and
\begin{alignat}{3}
\{\Lambda_n^+\}_n&\sim f,&\qquad\Lambda_n^+&=\{\lambda_{2i-1}(H_n(f))\}_{i=1,\ldots,\lceil n/2\rceil}=\{\lambda_{2i-1}(T_n(f))\}_{i=1,\ldots,\lceil n/2\rceil},\label{l+f+}\\
\{\Lambda_n^-\}_n&\sim-f,&\qquad\Lambda_n^-&=\{\lambda_{2i}(H_n(f))\}_{i=1,\ldots,\lfloor n/2\rfloor}=\{-\lambda_{2i}(T_n(f))\}_{i=1,\ldots,\lfloor n/2\rfloor}.\label{l-f-}
\end{alignat}
Note that if we permute the columns of $V_n$ and the eigenvalues of $T_n(f)$ and $H_n(f)$ through a same permutation $\tau$ of $\{1,\ldots,n\}$ such that $\tau$ maps odd indices to odd indices and even indices to even indices, then \eqref{a-}--\eqref{e} continue to hold.

By \eqref{l+f+}--\eqref{l-f-}, Theorem~\ref{Toep-th}, and the assumptions on $f$, the hypotheses of Theorem~\ref{thm:mmd} are satisfied for $f$ and $\Lambda_n^+$ as well as for $-f$ and $\Lambda_n^-$.
Hence, by Theorem~\ref{thm:mmd} applied first with $f$ and $\Lambda_n^+$ and then with $-f$ and $\Lambda_n^-$, we infer the existence of two a.u.\ grids $\{x^+_{i,n}\}_{i=1,\ldots,\lceil n/2\rceil}$, $\{x^-_{i,n}\}_{i=1,\ldots,\lfloor n/2\rfloor}$ in $[0,\pi]$ and two permutations $\tau_n^+$ of $\{1,\ldots,\lceil n/2\rceil\}$ and $\tau_n^-$ of $\{1,\ldots,\lfloor n/2\rfloor\}$ such that, for every~$n$,
\begin{alignat}{3}
\lambda_{2\tau_n^+(i)-1}(T_n(f))&=f(x_{i,n}^+),\qquad &i&=1,\ldots,\lceil n/2\rceil,\label{f(x)}\\
-\lambda_{2\tau_n^-(i)}(T_n(f))&=-f(x_{i,n}^-),\qquad &i&=1,\ldots,\lfloor n/2\rfloor.\label{-f(x)}
\end{alignat}
Define
\begin{align*}
\{x_{1,n},\ldots,x_{n,n}\}&=\{x_{1,n}^+,x_{1,n}^-,x_{2,n}^+,x_{2,n}^-,\ldots\}
\end{align*}
and rearrange the eigenvalues of $H_n(f)$ and $T_n(f)$ as follows:
\begin{align*}
\{\lambda_1(H_n(f)),\ldots,\lambda_n(H_n(f))\}&=\{\lambda_{2\tau_n^+(1)-1}(H_n(f)),\lambda_{2\tau_n^-(1)}(H_n(f)),\lambda_{2\tau_n^+(2)-1}(H_n(f)),\lambda_{2\tau_n^-(2)}(H_n(f)),\ldots\},\\
\{\lambda_1(T_n(f)),\ldots,\lambda_n(T_n(f))\}&=\{\lambda_{2\tau_n^+(1)-1}(T_n(f)),\lambda_{2\tau_n^-(1)}(T_n(f)),\lambda_{2\tau_n^+(2)-1}(T_n(f)),\lambda_{2\tau_n^-(2)}(T_n(f)),\ldots\}.
\end{align*}
It is easy to check that $\{x_{i,n}\}_{i=1,\ldots,n}$ is an a.u.\ grid in $[0,\pi]$.
Moreover, after the above rearrangement of the eigenvalues of $T_n(f)$ and $H_n(f)$, by \eqref{f(x)}--\eqref{-f(x)} we have
\[ \lambda_i(T_n(f))=f(x_{i,n}),\qquad i=1,\ldots,n, \]
which is \eqref{a}. In addition, \eqref{a-}--\eqref{e} continue to hold, because, by construction, the considered rearrangement is associated with a permutation $\tau_n$ of the columns of $V_n$ and the eigenvalues of $T_n(f)$ and $H_n(f)$ such that $\tau_n$ maps odd indices to odd indices and even indices to even indices. Of course, \eqref{a-}--\eqref{e} continue to hold with a new matrix $V_n$ obtained by permuting the columns of the old $V_n$ through the permutation $\tau_n$. With abuse of notation, we denote again by $V_n$ the new matrix $V_n$, so that \eqref{a-}--\eqref{e} hold unchanged. The thesis is proved.
\end{proof}

Theorem~\ref{cbt'} is our fourth main result. It is an extension of Corollary~\ref{cb} to the case where $f\in L^1([-\pi,\pi])$ is only assumed to have real Fourier coefficients. In this case, the moduli of the eigenvalues of $H_n(f)$ coincide with the singular values of $T_n(f)$, as shown by the following remark.

\begin{remark}\label{eig-singv}
For every matrix $A\in\mathbb C^{n\times n}$, the singular values of $A$ and $Y_nA$ coincide because $Y_n$ is a unitary (permutation) matrix. In particular, for every $f\in L^1([-\pi,\pi])$, the singular values of $T_n(f)$ and $H_n(f)$ coincide. In the case where $T_n(f)$ is real, which happens whenever the Fourier coefficients of $f$ are real, the matrix $H_n(f)$ is real and symmetric, and so the singular values of $H_n(f)$, i.e., the singular values of $T_n(f)$, coincide with the moduli of the eigenvalues of $H_n(f)$.
\end{remark}

In what follows, a matrix $S\in\mathbb C^{n\times n}$ is referred to as a phase matrix if $S$ is diagonal and $|S_{ii}|=1$ for all $i=1,\ldots,n$. Note that every phase matrix is unitary.

\begin{theorem}\label{cbt'}
Suppose that the Fourier coefficients of $f\in L^1([-\pi,\pi])$ are real. Then, for every $n$, the following properties hold. 
\begin{enumerate}[nolistsep,leftmargin=*]
	\item There exists an ordering of the singular values of $T_n(f)$ and eigenvalues of $H_n(f)$ such that $|\lambda_i(H_n(f))|=\sigma_i(T_n(f))$, $i=1,\ldots,n$.
	\item Suppose that $|\lambda_i(H_n(f))|=\sigma_i(T_n(f))$, $i=1,\ldots,n$, and let $\{\uu_1,\ldots,\uu_n\}$ be an orthonormal basis of $\mathbb R^n$ consisting of left singular vectors of $T_n(f)$ associated with $\sigma_1(T_n(f)),\ldots,\sigma_n(T_n(f))$, respectively. 
	Then $\{Y_n\uu_1,\ldots,Y_n\uu_n\}$ is an orthonormal basis of $\mathbb R^n$ consisting of eigenvectors of $H_n(f)^2$ associated with $\lambda_1(H_n(f))^2,\ldots,\lambda_n(H_n(f))^2$, respectively.
	\item Suppose that $|\lambda_i(H_n(f))|=\sigma_i(T_n(f))$, $i=1,\ldots,n$, and let $\{\vv_1,\ldots,\vv_n\}$ be an orthonormal basis of $\mathbb R^n$ consisting of right singular vectors of $T_n(f)$ associated with $\sigma_1(T_n(f)),\ldots,\sigma_n(T_n(f))$, respectively. 
	Then $\{\vv_1,\ldots,\vv_n\}$ is an orthonormal basis of $\mathbb R^n$ consisting of eigenvectors of $H_n(f)^2$ associated with $\lambda_1(H_n(f))^2,\ldots,\lambda_n(H_n(f))^2$, respectively. 
\end{enumerate}
\end{theorem}
\begin{proof}\hfill
\begin{enumerate}[leftmargin=*,nolistsep]
	\item This has been proved in Remark~\ref{eig-singv}.
	\item Suppose that $|\lambda_i(H_n(f))|=\sigma_i(T_n(f))$, $i=1,\ldots,n$, and let $\{\uu_1,\ldots,\uu_n\}$ be an orthonormal basis of $\mathbb R^n$ consisting of left singular vectors of $T_n(f)$ associated with $\sigma_1(T_n(f)),\ldots,\sigma_n(T_n(f))$, respectively. Set $U_n=\left[\uu_1\,|\,\cdots\,|\,\uu_n\right]$ and let $T_n(f)=U_n\Sigma_nV_n^*$ be a singular value decomposition of $T_n(f)$ with left singular vectors $\uu_1,\ldots,\uu_n$, where $V_n$ is unitary like $U_n$ and $\Sigma_n$ is the diagonal matrix whose diagonal elements are the singular values $\sigma_1(T_n(f)),\ldots,\sigma_n(T_n(f))$.
	Let $S_n$ be a phase matrix such that $\Sigma_n=S_n\Lambda_n$, where 
	\[ \Lambda_n=\mathop{\rm diag}_{i=1,\ldots,n}\lambda_i(H_n(f)). \]
	Then,
	\[ H_n(f)=Y_nT_n(f)=Y_nU_n\Sigma_nV_n^*=Y_nU_nS_n\Lambda_nV_n^*. \]
	Since $H_n(f)$ is real and symmetric, we have
	\begin{align*}
	H_n(f)^2&=H_n(f)H_n(f)^*=Y_nU_nS_n\Lambda_nV_n^*V_n\Lambda_n^*S_n^*U_n^*Y_n^*=Y_nU_n\Lambda_n^2(Y_nU_n)^*.
	\end{align*}
	This means that $Y_n\uu_i$ is an eigenvector of $H_n(f)^2$ associated with the eigenvalue $\lambda_i(H_n(f))^2$ for all $i=1,\ldots,n$.
	\item The proof of property~3 is analogous to the proof of property~2. The details are left to the reader. \qedhere
\end{enumerate}
\end{proof}

\begin{remark}
For every diagonalizable matrix $A\in\mathbb C^{n\times n}$, the eigenvectors of $A$ and $A^2$ coincide whenever the eigenvalues of $A^2$ (i.e., the squares of the eigenvalues of $A$) are distinct. More precisely, in this case we have that $\vv$ is an eigenvector of $A$ associated with the eigenvalue $\lambda$ if and only if $\vv$ is an eigenvector of $A^2$ associated with $\lambda^2$. It follows that, {\em in items~{\rm2} and~{\rm3} of Theorem~{\rm\ref{cbt'}}, we can replace ``$H_n(f)^2$'' with ``$H_n(f)$'' and ``$\lambda_1(H_n(f))^2,\ldots,\lambda_n(H_n(f))^2$'' with ``$\lambda_1(H_n(f)),\ldots,\lambda_n(H_n(f))$'' whenever the eigenvalues of $H_n(f)^2$ are distinct.}
\end{remark}

Theorem~\ref{cb'} is our fifth main result. To prove it, we need two auxiliary lemmas. The first one is a plain extension of Widom's Lemma~\ref{Widom's}. The second one combines a few classical results on Toeplitz matrices \cite{GLTbookI}.

\begin{lemma}\label{Widom's'}
Let $f\in L^1([-\pi,\pi])$ and let $d$ be the distance of the complex zero from the convex hull of the essential range $\mathcal{ER}(f)$. Suppose that $|f|$ is not a.e.\ constant. Then, the singular values of $T_n(f)$ lie in $[d,M_{|f|})$ for all $n$, where $M_{|f|}=\mathop{\rm ess\,sup}_{[-\pi,\pi]}|f|$.
\end{lemma}
\begin{proof}
By Widom's Lemma~\ref{Widom's}, the singular values of $T_n(f)$ lie in $[d,M_{|f|}]$ for all $n$. We show that no singular value of $T_n(f)$ can be equal to $M_{|f|}$, i.e., $\|T_n(f)\|<M_{|f|}$, where $\|T_n(f)\|$ is the spectral (or Euclidean) norm of $T_n(f)$ (the largest singular value of $T_n(f)$). It is known that $\|T_n(f)\|\le\|T_n(|f|)\|$; see, e.g., \cite[Lemma~6.3]{GLTbookI} applied with $p=\infty$. By Theorem~\ref{Toep-th} and the assumption that $|f|$ is not a.e.\ constant, we infer that $T_n(|f|)$ is a Hermitian positive definite matrix whose eigenvalues lie in $(m_{|f|},M_{|f|})\subseteq(0,M_{|f|})$. In particular, the largest eigenvalue of $T_n(|f|)$ coincides with $\|T_n(|f|)\|$ and is smaller that $M_{|f|}$. Thus, $\|T_n(f)\|\le\|T_n(|f|)\|<M_{|f|}$.
\end{proof}

\begin{lemma}\label{cr-book}
Let $f(\theta)=\sum_{k=-r}^rf_k{\rm e}^{{\rm i}k\theta}$ be a trigonometric polynomial of degree $r$, and let
\[ m_{|f|}=\min_{[-\pi,\pi]}|f|,\qquad M_{|f|}=\max_{[-\pi,\pi]}|f|. \]
Then, for every $n$, the singular values of $T_n(f)$ lie in $[m_{|f|},M_{|f|}]$ except for at most $2r$ outliers smaller than $m_{|f|}$. 
\end{lemma}
\begin{proof}
The thesis follows immediately from Lemma~\ref{Widom's} for $n\le2r$.
Suppose that $n\ge2r+1$. Let $C_n(f)$ be the circulant matrix defined in \cite[p.~109]{GLTbookI}. By the first inequality in \cite[p.~110]{GLTbookI}, we have
\[ {\rm rank}(T_n(f)-C_n(f))\le 2r. \]
Hence, by the interlacing theorem for singular values \cite[Theorem~2.11]{GLTbookI}, the singular values of $T_n(f)$ lie between $\sigma_{\min}(C_n(f))$ and $\sigma_{\max}(C_n(f))$, except for at most $2r$ singular values smaller than $\sigma_{\min}(C_n(f))$ and $2r$ singular values larger than $\sigma_{\max}(C_n(f))$ (which are anyway $\le M_{|f|}$ by Lemma~\ref{Widom's}). We know from \cite[Theorem~6.4]{GLTbookI} that $\sigma_{\min}(C_n(f))$ and $\sigma_{\max}(C_n(f))$ lie between $[m_{|f|},M_{|f|}]$. 
Thus, all the singular values of $T_n(f)$ lie in $[m_{|f|},M_{|f|}]$ except for at most $2r$ outliers smaller than $m_{|f|}$.
\end{proof}

\begin{theorem}\label{cb'}
Suppose that the Fourier coefficients of $f\in L^1([-\pi,\pi])$ are real. 
Then, for every $n$, the following properties hold.
\begin{enumerate}[nolistsep,leftmargin=*]
	\item The eigenvalues of $H_n(f)$ lie in $[-M_{|f|},-d]\cup[d,M_{|f|}]$, where $d$ is the distance of the complex zero from the convex hull of the essential range $\mathcal{ER}(f)$ and $M_{|f|}=\mathop{\rm ess\,sup}_{[-\pi,\pi]}|f|$.
	\item Assume that $|f|$ is not a.e.\ constant. Then, the eigenvalues of $H_n(f)$ lie in $(-M_{|f|},-d]\cup[d,M_{|f|})$, where $d$ and $M_{|f|}$ are as in item~1.
	\item Assume that $f(\theta)=\sum_{k=-r}^rf_k{\rm e}^{{\rm i}k\theta}$ is a trigonometric polynomial of degree $r$, and let
	\[ m_{|f|}=\min_{[-\pi,\pi]}|f|,\qquad M_{|f|}=\max_{[-\pi,\pi]}|f|. \]
	Then, the eigenvalues of $H_n(f)$ lie in $[-M_{|f|},-m_{|f|}]\cup[m_{|f|},M_{|f|}]$ except for at most $2r$ small outliers lying in $(-m_{|f|},m_{|f|})$. 
\end{enumerate}
\end{theorem}
\begin{proof}
\begin{enumerate}[nolistsep,leftmargin=*]
	\item We know that $|\lambda_i(H_n(f))|=\sigma_i(T_n(f))$, $i=1,\ldots,n$. Hence, the result follows from Lemma~\ref{Widom's}. 
	\item We know that $|\lambda_i(H_n(f))|=\sigma_i(T_n(f))$, $i=1,\ldots,n$. Hence, the result follows from Lemma~\ref{Widom's'}. 
	\item We know that $|\lambda_i(H_n(f))|=\sigma_i(T_n(f))$, $i=1,\ldots,n$. Hence, the result follows from Lemma~\ref{cr-book}. \qedhere
\end{enumerate}
\end{proof}

Our last main result (Theorem~\ref{psi-thm}) is more an observation than a ``main result'', but we decided anyway to state it here, in the section of main results, as it completes our spectral study of flipped Toeplitz matrices. To prove Theorem~\ref{psi-thm}, we need the following basic lemmas, which can be seen as corollaries of \cite[Proposition~3.1.2]{Grafakos}; see also \cite[Exercise~4.5]{Vretblad}. For the reader's convenience, we include the short proofs.

\begin{lemma}\label{feven}
Let $f\in L^1([-\pi,\pi])$. Then, the following are equivalent.
\begin{enumerate}[nolistsep,leftmargin=*]
	\item $f$ is real and even a.e.\ in $[-\pi,\pi]$, i.e., $f(-x)=f(x)\in\mathbb R$ for a.e.\ $x\in[-\pi,\pi]$.
	\item The Fourier coefficients of $f$ are real and even, i.e., $f_{-k}=f_k\in\mathbb R$ for all $k\in\mathbb Z$.
\end{enumerate}
\end{lemma}
\begin{proof}
$(1\implies2)$ Suppose that $f(-x)=f(x)\in\mathbb R$ for almost every $x\in[-\pi,\pi]$. Then, for every $k\in\mathbb Z$,
\begin{align*}
f_{-k}&=\frac1{2\pi}\int_{-\pi}^\pi f(x){\rm e}^{{\rm i}kx}{\rm d}x=\frac1{2\pi}\int_{-\pi}^\pi f(-t){\rm e}^{-{\rm i}kt}{\rm d}t=\frac1{2\pi}\int_{-\pi}^\pi f(t){\rm e}^{-{\rm i}kt}{\rm d}t=f_k
\end{align*}
and $f_k$ is real, because
\begin{align*}
f_k&=\frac12(f_{-k}+f_k)=\frac12\left(\frac1{2\pi}\int_{-\pi}^\pi f(x){\rm e}^{{\rm i}kx}{\rm d}x+\frac1{2\pi}\int_{-\pi}^\pi f(x){\rm e}^{-{\rm i}kx}{\rm d}x\right)=\frac1{2\pi}\int_{-\pi}^\pi f(x)\cos(kx){\rm d}x.
\end{align*}
$(2\implies1)$ Suppose that the Fourier coefficients of $f$ are real and even. In order to prove that $f$ is real and even a.e.\ in $[-\pi,\pi]$, it suffices to prove that the three functions $f(-x)$, $f(x)$, $\overline{f(x)}$ have the same Fourier coefficients, which means that they coincide a.e.\ \cite[Theorem~5.15]{Rudinone}. Let $\{a_k\}_{k\in\mathbb Z}$ (resp., $\{b_k\}_{k\in\mathbb Z}$, $\{f_k\}_{k\in\mathbb Z}$) be the sequence of Fourier coefficients of $f(-x)$ (resp., $\overline{f(x)}$, $f(x)$). Then, taking into account that the Fourier coefficients of $f$ are real and even, for every $k\in\mathbb Z$ we have
\begin{align*}
b_k&=\frac1{2\pi}\int_{-\pi}^\pi\overline{f(x)}{\rm e}^{-{\rm i}kx}{\rm d}x=\overline{\frac1{2\pi}\int_{-\pi}^\pi f(x){\rm e}^{{\rm i}kx}{\rm d}x}=\overline{f_{-k}}=f_{-k}=\frac1{2\pi}\int_{-\pi}^\pi f(x){\rm e}^{{\rm i}kx}{\rm d}x=\frac1{2\pi}\int_{-\pi}^\pi f(-t){\rm e}^{-{\rm i}kt}{\rm d}t=a_k,
\end{align*}
hence $b_k=a_k=f_{-k}=f_k$.
\end{proof}

\begin{lemma}\label{mod(f)even}
Suppose that the Fourier coefficients of $f\in L^1([-\pi,\pi])$ are real. Then, $f(-x)=\overline{f(x)}$ for almost every $x\in[-\pi,\pi]$.
\end{lemma}
\begin{proof} 
It suffices to prove that the two functions $f(-x)$ and $\overline{f(x)}$ have the same Fourier coefficients, which means that they coincide a.e.\ \cite[Theorem~5.15]{Rudinone}. Let $\{a_k\}_{k\in\mathbb Z}$ (resp., $\{b_k\}_{k\in\mathbb Z}$, $\{f_k\}_{k\in\mathbb Z}$) be the sequence of Fourier coefficients of $f(-x)$ (resp., $\overline{f(x)}$, $f(x)$). Then, taking into account that the Fourier coefficients of $f$ are real, for every $k\in\mathbb Z$ we have
\begin{align*}
b_k&=\frac1{2\pi}\int_{-\pi}^\pi\overline{f(x)}{\rm e}^{-{\rm i}kx}{\rm d}x=\overline{\frac1{2\pi}\int_{-\pi}^\pi f(x){\rm e}^{{\rm i}kx}{\rm d}x}=\overline{f_{-k}}=f_{-k}=\frac1{2\pi}\int_{-\pi}^\pi f(x){\rm e}^{{\rm i}kx}{\rm d}x=\frac1{2\pi}\int_{-\pi}^\pi f(-t){\rm e}^{-{\rm i}kt}{\rm d}t=a_k,
\end{align*}
hence $b_k=a_k$.
\end{proof}

To simplify the statement of Theorem~\ref{psi-thm}, we borrow a notation from \cite{SIMAX-hankel}: for every $g:[0,\pi]\to\mathbb C$, we define the function $\psi_g:[0,2\pi]\to\mathbb C$ by setting
\[ \psi_g(x)=\left\{\begin{aligned}&g(x), &x&\in[0,\pi],\\
&\!-g(x-\pi), &x&\in(\pi,2\pi].
\end{aligned}\right. \]

\begin{theorem}\label{psi-thm}
Let $f\in L^1([-\pi,\pi])$. 
Then, the following properties hold.
\begin{enumerate}[nolistsep,leftmargin=*]
	\item Suppose that the Fourier coefficients of $f$ are real. Then,
	\[ \{T_n(f)\}_n\sim_\sigma f|_{[0,\pi]},\qquad \{H_n(f)\}_n\sim_\sigma f|_{[0,\pi]}. \]
	\item Suppose that the Fourier coefficients of $f$ are real and even.
	Then,
	\[ \{T_n(f)\}_n\sim_\lambda f|_{[0,\pi]},\qquad\{H_n(f)\}_n\sim_\lambda\psi_{f|_{[0,\pi]}},\qquad\{H_n(f)\}_n\sim_\lambda\psi_{|f|_{[0,\pi]}|}. \]
\end{enumerate}
\end{theorem}
\begin{proof}\hfill
\begin{enumerate}[nolistsep,leftmargin=*]
	\item By Lemma~\ref{mod(f)even}, we have $f(-x)=\overline{f(x)}$ for almost every $x\in[-\pi,\pi]$, hence $|f(-x)|=|f(x)|$ for almost every $x\in[-\pi,\pi]$. Thus, the relation $\{T_n(f)\}_n\sim_\sigma f|_{[0,\pi]}$ is a consequence of the relation $\{T_n(f)\}_n\sim_\sigma f$ (which holds by Theorem~\ref{Toep-th}) and the definition of asymptotic singular value distribution. The relation $\{H_n(f)\}_n\sim_\sigma f|_{[0,\pi]}$ follows immediately from $\{T_n(f)\}_n\sim_\sigma f|_{[0,\pi]}$ and the fact that the singular values of $H_n(f)$ and $T_n(f)$ coincide; see Remark~\ref{eig-singv}.
	\item By Lemma~\ref{feven}, $f$ is real and even a.e. Thus, the relation $\{T_n(f)\}_n\sim_\lambda f|_{[0,\pi]}$ is a consequence of the relation $\{T_n(f)\}_n\sim_\lambda f$ (which holds by Theorem~\ref{Toep-th}) and the definition of asymptotic spectral distribution. The relation $\{H_n(f)\}_n\sim_\lambda\psi_{f|_{[0,\pi]}}$ is a consequence of the relation
	\begin{equation}\label{ehi}
	\{H_n(f)\}_n\sim_\lambda{\rm diag}(f(x),-f(x)),\qquad x\in[0,\pi],
	\end{equation}
	(which holds by Theorem~\ref{r-hankel} and the evenness of $f$) and Lemma~\ref{lem:concat} applied with $[\alpha,\beta]=[0,2\pi]$.
	Finally, the relation $\{H_n(f)\}_n\sim_\lambda\psi_{|f|_{[0,\pi]}|}$ is a consequence of the relation $\{H_n(f)\}_n\sim_\lambda\psi_{f|_{[0,\pi]}}$ and the definition of asymptotic spectral distribution. \qedhere
\end{enumerate}
\end{proof}

\section{Numerical experiments}\label{num_exps}

We present in this section a few numerical examples that illustrate Theorems~\ref{thm3-Giov-vts} and~\ref{0notin}. Note that the thesis of Theorem~\ref{0notin} is stronger than the thesis of Theorem~\ref{thm3-Giov-vts}, because the existence of an a.u.\ grid $\{x_{i,n}\}_{i=1,\ldots,n}$ in $[0,\pi]$ such that \eqref{a} is satisfied implies that \eqref{af} is satisfied for every a.u.\ grid $\{x_{i,n}\}_{i=1,\ldots,n}$ in $[0,\pi]$. Thus, for functions $f$ satisfying the hypotheses of both Theorems~\ref{thm3-Giov-vts} and~\ref{0notin}, we just illustrate Theorem~\ref{0notin}.

\begin{example}\label{e2}

\begin{figure}
\centering
\begin{minipage}{0.485\textwidth}
\centering
\includegraphics[width=\textwidth]{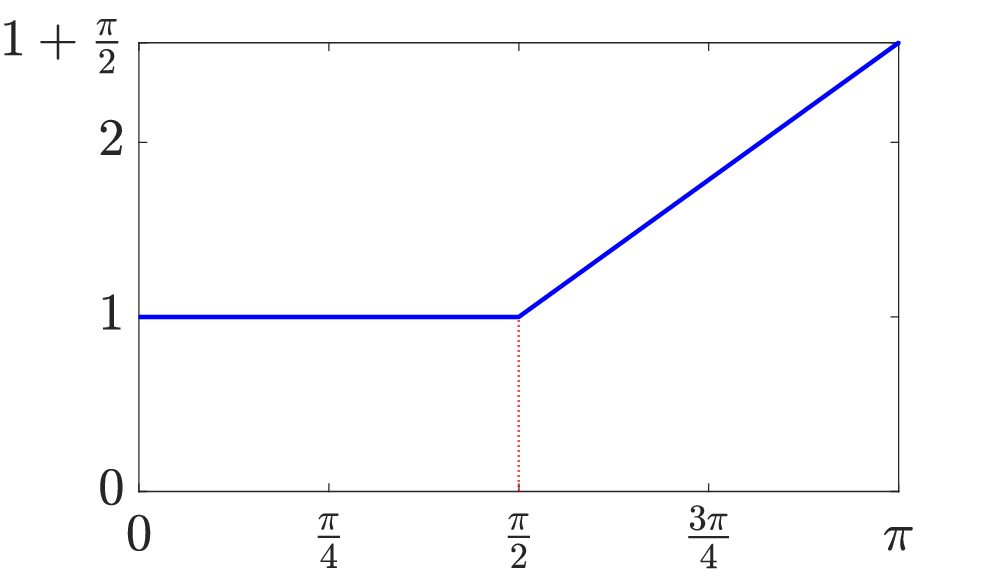}
\caption{Example~\ref{e2}: Graph on the interval $[0,\pi]$ of the function $f(\theta)$ defined in \eqref{fc}.}
\label{fc_graph}
\end{minipage}
\hspace{0.01\textwidth}
\begin{minipage}{0.485\textwidth}
\centering
\captionof{table}{Example~\ref{e2}: Computation of $M_n$ for increasing values of $n$.}
\begin{tabular}{rc}
\toprule
$n$ & $M_n$\\
\midrule
8   & 0.0851\\
16  & 0.0632\\
32  & 0.0454\\
64  & 0.0312\\
128 & 0.0206\\
256 & 0.0132\\
512 & 0.0082\\
1024 & 0.0050\\
\bottomrule
\end{tabular}
\label{e2t}
\end{minipage}
\end{figure}

Let $f:[-\pi,\pi]\to\mathbb R$,
\begin{equation}\label{fc}
f(\theta)=\left\{\begin{aligned}
&1, &\quad0&\le\theta<\pi/2,\\
&\theta+1-\pi/2, &\quad\pi/2&\le\theta\le\pi,\\
&f(-\theta), &\quad-\pi&\le\theta<0.
\end{aligned}\right.
\end{equation}
Figure~\ref{fc_graph} shows the graph of $f$ over the interval $[0,\pi]$. The function $f$ satisfies the hypotheses of Theorem~\ref{thm3-Giov-vts}.
Hence, by Theorem~\ref{thm3-Giov-vts}, for every $n$ and every a.u.\ grid $\{x_{i,n}\}_{i=1,\ldots,n}$ in $[0,\pi]$, there exist a real unitary matrix $V_n$ and an ordering of the eigenvalues of $T_n(f)$ and $H_n(f)$ such that
\begin{alignat}{3}
Y_n&=V_n\Delta_nV_n^*,&\qquad\Delta_n&=\mathop{\rm diag}_{i=1,\ldots,n}(-1)^{i+1},\label{a-e2}\\
T_n(f)&=V_nD_nV_n^*,&\qquad D_n&=\mathop{\rm diag}_{i=1,\ldots,n}\lambda_i(T_n(f)),\label{a--e2}\\
H_n(f)&=V_nE_nV_n^*,&\qquad E_n&=\Delta_nD_n=\mathop{\rm diag}_{i=1,\ldots,n}\lambda_i(H_n(f)),\label{a---e2}\\
\lambda_i(H_n(f))&=(-1)^{i+1}\lambda_i(T_n(f)), &\qquad i&=1,\ldots,n,\label{ee2}\\
M_n&=\max_{i=1,\ldots,n}|f(x_{i,n})-\lambda_i(T_n(f))|\to0\ \,\mbox{as}\,\ n\to\infty.\label{ae2}
\end{alignat}
To provide numerical evidence of this, for the values of $n$ considered in Table~\ref{e2t}, we arranged the eigenvalues of $T_n(f)$ so that \eqref{a-e2}--\eqref{ee2} are satisfied. In other words, we arranged the eigenvalues of $T_n(f)$ so that the eigenvector associated with the $i$th eigenvalue $\lambda_i(T_n(f))$ is either symmetric or skew-symmetric depending on whether $i$ is odd or even. Then, we computed $M_n$ in the case of the a.u.\ grid $x_{i,n}=\frac{i\pi}{n+1}$, $i=1,\ldots,n$.
We see from the table that $M_n\to0$ as $n\to\infty$, though the convergence is slow.

Now we observe that $f$ does not satisfy the hypotheses of Theorem~\ref{0notin}. Actually, $f$ satisfies all the hypotheses of Theorem~\ref{0notin} except the assumption that $f$ has a finite number of local maximum/minimum points. Indeed, $f$ is constant on $[0,\pi/2]$ and so all points in $[0,\pi/2)$ are both local maximum and local minimum points for $f$ according to our Definition~\ref{wlep}. We observe that, in fact, the thesis of Theorem~\ref{0notin} does not hold in this case, because there is no a.u.\ grid $\{x_{i,n}\}_{i=1,\ldots,n}$ in $[0,\pi]$ such that, for every $n$,
\[ \lambda_i(T_n(f))=f(x_{i,n}),\qquad i=1,\ldots,n, \]
for a suitable ordering of the eigenvalues of $T_n(f)$. Indeed, the eigenvalues of $T_n(f)$ are contained in $(1,1+\pi/2)$ by Theorem~\ref{Toep-th} and so any grid $\{x_{i,n}\}_{i=1,\ldots,n}\subset[0,\pi]$ satisfying the previous condition must be contained in $(\pi/2,\pi)$, which implies that it cannot be a.u.\ in $[0,\pi]$.
\end{example}

\begin{example}\label{e2.5}
This example is suggested by the cubic B-spline Galerkin discretization of second-order eigenvalue (and Poisson) problems \cite[Section~2.4.1]{Tom-paper}.
Let $f:[-\pi,\pi]\to\mathbb R$,
\begin{equation}\label{f-IgA}
f(\theta)=\frac23-\frac14\cos(\theta)-\frac25\cos(2\theta)-\frac1{60}\cos(3\theta).
\end{equation}
Figure~\ref{f-IgA_graph} shows the graph of $f$ over the interval $[0,\pi]$.
The function $f$ satisfies the hypotheses of Theorem~\ref{0notin}. Hence, by Theorem~\ref{0notin}, for every $n$ there exist an a.u.\ grid $\{x_{i,n}\}_{i=1,\ldots,n}$ in $[0,\pi]$, a real unitary matrix $V_n$, and an ordering of the eigenvalues of $T_n(f)$ and $H_n(f)$ such that
\begin{alignat}{3}
Y_n&=V_n\Delta_nV_n^*,&\qquad\Delta_n&=\mathop{\rm diag}_{i=1,\ldots,n}(-1)^{i+1},\label{a-e1-IgA}\\
T_n(f)&=V_nD_nV_n^*,&\qquad D_n&=\mathop{\rm diag}_{i=1,\ldots,n}\lambda_i(T_n(f)),\label{a--e1-IgA}\\
H_n(f)&=V_nE_nV_n^*,&\qquad E_n&=\Delta_nD_n=\mathop{\rm diag}_{i=1,\ldots,n}\lambda_i(H_n(f)),\label{a---e1-IgA}\\
\lambda_i(H_n(f))&=(-1)^{i+1}\lambda_i(T_n(f)), &\qquad i&=1,\ldots,n,\label{ee1-IgA}\\
\lambda_i(T_n(f))&=f(x_{i,n}), &\qquad i&=1,\ldots,n.\label{ae1-IgA}
\end{alignat}
To provide numerical evidence of this, for the values of $n$ considered in Table~\ref{e2.5t}, we arranged the eigenvalues of $T_n(f)$ so that \eqref{a-e1-IgA}--\eqref{ee1-IgA} are satisfied. In other words, we arranged the eigenvalues of $T_n(f)$ so that the eigenvector associated with the $i$th eigenvalue $\lambda_i(T_n(f))$ is either symmetric or skew-symmetric depending on whether $i$ is odd or even. Then, we computed a grid $\mathcal G_n=\{x_{i,n}\}_{i=1,\ldots,n}$ satisfying \eqref{ae1-IgA} and we reported in Table~\ref{e2.5t} its uniformity measure $m(\mathcal G_n)$. We see from the table that $m(\mathcal G_n)\to0$ as $n\to\infty$, meaning that $\mathcal G_n$ is a.u.\ in $[0,\pi]$, though the convergence to $0$ of $m(\mathcal G_n)$ is slow.

\begin{figure}
\centering
\begin{minipage}{0.485\textwidth}
\centering
\includegraphics[width=\textwidth]{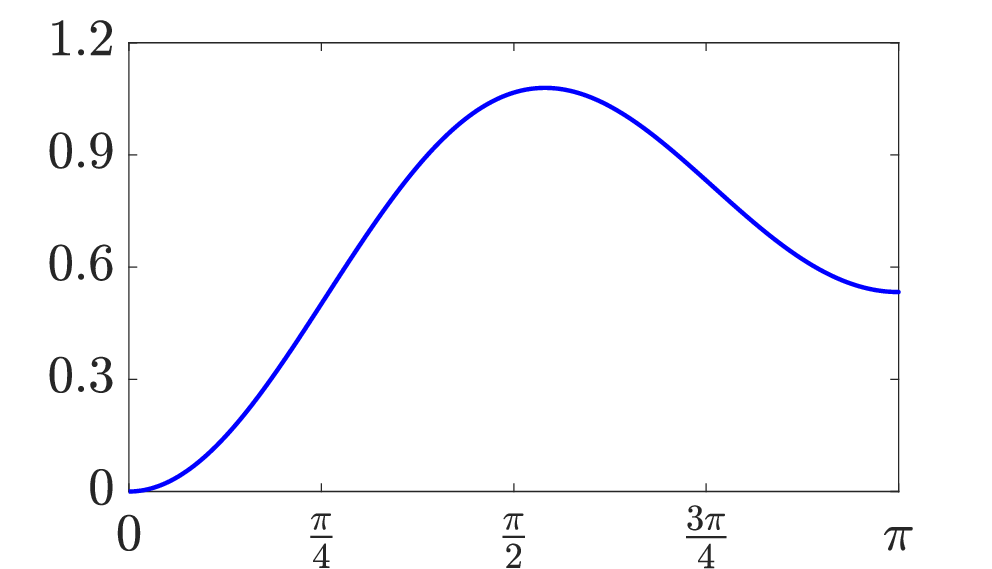}
\caption{Example~\ref{e2.5}: Graph on the interval $[0,\pi]$ of the function $f(\theta)$ defined in \eqref{f-IgA}.}
\label{f-IgA_graph}
\end{minipage}
\hspace{0.01\textwidth}
\begin{minipage}{0.485\textwidth}
\centering
\captionof{table}{Example~\ref{e2.5}: Computation of $m(\mathcal G_n)$ for increasing values of $n$.}
\begin{tabular}{rcc}
\toprule
$n$ & $m(\mathcal G_n)$\\
\midrule
8   & 0.2475\\
16  & 0.1304\\
32  & 0.0753\\
64  & 0.0403\\
128 & 0.0328\\
256 & 0.0245\\
512 & 0.0119\\
1024 & 0.0060\\
\bottomrule
\end{tabular}
\label{e2.5t}
\end{minipage}
\end{figure}

\end{example}

\begin{example}\label{e3}
In this last example, we consider a discontinuous function.
Let $f:[-\pi,\pi]\to\mathbb R$,
\begin{equation}\label{fd}
f(\theta)=\left\{\begin{aligned}
&\cos(2\theta)+\cos(3\theta), &\quad0&\le\theta<\pi/2,\\
&\theta, &\quad\pi/2&\le\theta\le\pi,\\
&f(-\theta), &\quad-\pi&\le\theta<0.
\end{aligned}\right.
\end{equation}
Figure~\ref{fd_graph} shows the graph of $f$ over the interval $[0,\pi]$.
The function $f$ satisfies the hypotheses of Theorem~\ref{0notin} just as the function $f$ of Example~\ref{e2.5}. Hence, we proceeded exactly as in Example~\ref{e2.5}. The results are collected in Table~\ref{e3t}, which is the version of Table~\ref{e2.5t} for this example.

\begin{figure}
\centering
\begin{minipage}{0.485\textwidth}
\centering
\includegraphics[width=\textwidth]{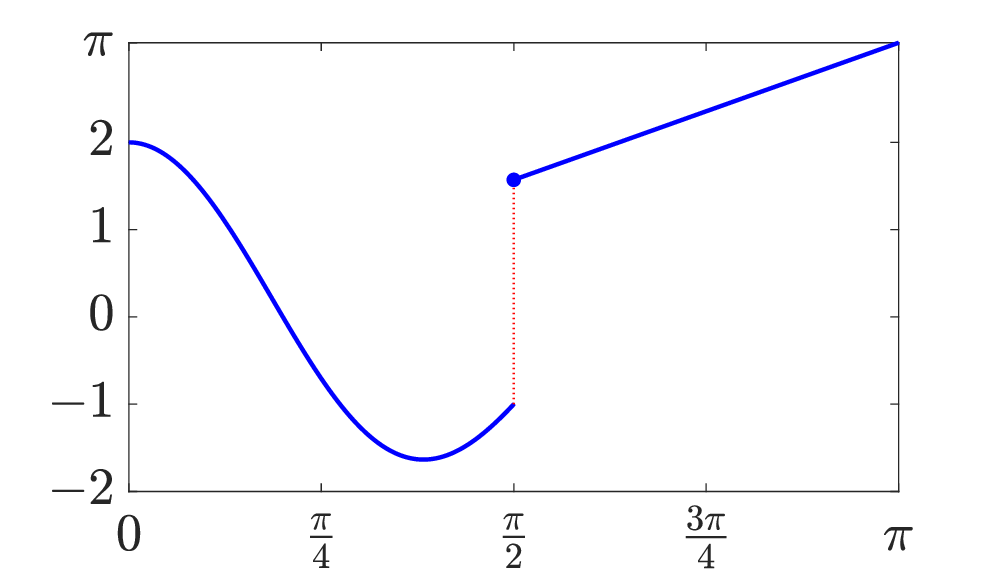}
\caption{Example~\ref{e3}: Graph on the interval $[0,\pi]$ of the function $f(\theta)$ defined in \eqref{fd}.}
\label{fd_graph}
\end{minipage}
\hspace{0.01\textwidth}
\begin{minipage}{0.485\textwidth}
\centering
\captionof{table}{Example~\ref{e3}: Computation of $m(\mathcal G_n)$ for increasing values of $n$.}
\begin{tabular}{rcc}
\toprule
$n$ & $m(\mathcal G_n)$\\
\midrule
8   & 0.5771\\
16  & 0.4633\\
32  & 0.2421\\
64  & 0.2082\\
128 & 0.1127\\
256 & 0.0812\\
512 & 0.0336\\
1024 & 0.0183\\
\bottomrule
\end{tabular}
\label{e3t}
\end{minipage}
\end{figure}

\end{example}

\section{Conclusions}\label{conc}
We have studied the spectral properties of flipped Toeplitz matrices of the form $H_n(f)=Y_nT_n(f)$, where $T_n(f)$ is the $n\times n$ Toeplitz matrix generated by $f$ and $Y_n$ is the exchange (flip) matrix in \eqref{flipM}. Our spectral results are collected in Theorems~\ref{cb+-}--\ref{psi-thm}.
The spectral properties obtained in this paper can be used in the convergence analysis of MINRES for the solution of real non-symmetric Toeplitz linear systems of the form $T_n(f)\mathbf x=\mathbf b$ after pre-multiplication of both sides by $Y_n$, as suggested by Pestana and Wathen \cite{PW}.

\section*{Acknowledgements}
Giovanni Barbarino, Carlo Garoni and Stefano Serra-Capizzano are members of the Research Group GNCS (Gruppo Nazionale per il Calcolo Scientifico) of INdAM (Istituto Nazionale di Alta Matematica).
Giovanni Barbarino acknowledges the support by the European Union (ERC consolidator grant, eLinoR, No 101085607).
Carlo Garoni was supported by the Department of Mathematics of the University of Rome Tor Vergata through the MUR Excellence Department Project MatMod@TOV (CUP E83C23000330006) and the Project RICH\hspace{1pt}\rule{5pt}{0.4pt}GLT (CUP E83C22001650005).
David Meadon was funded by the Centre for Interdisciplinary Mathematics (CIM) at Uppsala University.
Stefano Serra-Capizzano was funded by the European High-Performance Computing Joint Undertaking (JU) under grant agreement No 955701. The JU receives support from the European Union's Horizon 2020 research and innovation programme and Belgium, France, Germany, Switzerland.
Stefano Serra-Capizzano is also grateful to the Theory, Economics and Systems Laboratory (TESLAB) of the Department of Computer Science at the Athens University of Economics and Business for providing financial support.


\begin{thebibliography}{99}

\footnotesize

\bibitem{sdau}
{\sc Barbarino G., Ekstr\"om S.-E., Garoni C., Meadon D., Serra-Capizzano S., Vassalos P.}
{\em From asymptotic distribution and vague convergence to uniform convergence, with numerical applications.}
arXiv:2309.03662v1.

\bibitem{GLTbookIII}
{\sc Barbarino G., Garoni C., Serra-Capizzano S.}
{\em Block generalized locally Toeplitz sequences: theory and applications in the unidimensional case.}
Electron. Trans. Numer. Anal. 53 (2020) 28--112.

\bibitem{Fabio-book}
{\sc Bertaccini D., Durastante F.}
{\em Iterative Methods and Preconditioning for Large and Sparse Linear Systems with Applications.}
Taylor \& Francis, Boca Raton (2018).

\bibitem{BoSi}
{\sc B\"ottcher A., Silbermann B.}
{\em Introduction to Large Truncated Toeplitz Matrices.}
Springer, New York (1999).

\bibitem{SC-matrix}
{\sc Cantoni A., Butler P.}
{\em Eigenvalues and eigenvectors of symmetric centrosymmetric matrices.}
Linear Algebra Appl. 13 (1976) 275--288.

\bibitem{SIMAX-hankel}
{\sc Ferrari P., Furci I., Hon S., Mursaleen M. A., Serra-Capizzano S.}
{\em The eigenvalue distribution of special $2$-by-$2$ block matrix-sequences with applications to the case of symmetrized Toeplitz structures.}
SIAM J. Matrix Anal. Appl. 40 (2019) 1066--1086.

\bibitem{ELA-multihankel}
{\sc Ferrari P., Furci I., Serra-Capizzano S.}
{\em Multilevel symmetrized Toeplitz structures and spectral distribution results for the related matrix sequences.}
Electron. J. Linear Algebra 37 (2021) 370--386.

\bibitem{GLTbookI}
{\sc Garoni C., Serra-Capizzano S.}
{\em Generalized Locally Toeplitz Sequences: Theory and Applications (Volume I).}
Springer, Cham (2017).

\bibitem{Tom-paper}
{\sc Garoni C., Speleers H., Ekstr\"om S.-E., Reali A., Serra-Capizzano S., Hughes T. J. R.}
{\em Symbol-based analysis of finite element and isogeometric B-spline discretizations of eigenvalue problems: exposition and review.}
Arch. Comput. Methods Engrg. 26 (2019) 1639--1690.

\bibitem{Grafakos}
{\sc Grafakos L.}
{\em Classical Fourier Analysis.}
Third Edition, Springer, New York (2014).

\bibitem{BIT-hankel}
{\sc Mazza M., Pestana J.}
{\em Spectral properties of flipped Toeplitz matrices and related preconditioning.}
BIT Numer. Math. 59 (2019) 463--482.

\bibitem{SIMAX-multihankel}
{\sc Mazza M., Pestana J.}
{\em The asymptotic spectrum of flipped multilevel Toeplitz matrices and of certain preconditionings.}
SIAM J. Matrix Anal. Appl. 42 (2021) 1319--1336.

\bibitem{PW}
{\sc Pestana J., Wathen J.}
{\em A preconditioned MINRES method for nonsymmetric Toeplitz matrices.}
SIAM J. Matrix Anal. Appl. 36 (2015) 273--288.

\bibitem{Rudinone}
{\sc Rudin W.}
{\em Real and Complex Analysis.}
Third Edition, McGraw-Hill, Singapore (1987).

\bibitem{Vretblad}
{\sc Vretblad A.}
{\em Fourier Analysis and Its Applications.}
Springer, New York (2003).

\bibitem{Widom}
{\sc Widom H.}
{\em On the singular values of Toeplitz matrices.}
Zeitschrift Anal. Anwendung 8 (1989) 221--229.

\end{thebibliography}
\end{document}